\documentclass[3p,preprint]{elsarticle}
\usepackage[T1]{fontenc}
\usepackage{amsfonts}
\usepackage{amsmath}
\usepackage{geometry}
\usepackage{xcolor}
\usepackage{blindtext}
\usepackage{csquotes}
\usepackage{tikz}
\usepackage{tkz-euclide}
\usepackage{color}
\usepackage{placeins}
\usepackage{algorithm}      
\usepackage{algpseudocode}

\newtheorem{theorem}{Theorem}[section]

\newtheorem{lemma}[theorem]{Lemma}

\newtheorem{remark}[theorem]{Remark}

\newenvironment{proof}[1][Proof]{\noindent\textbf{#1.} }{\ \rule{0.5em}{0.5em}}
\biboptions{sort&compress}
\newcommand{\B}[1]{\boldsymbol{#1}}

\usepackage{amsfonts}
\usepackage{amsmath}
\usepackage{geometry}
\usepackage{csquotes}
\usepackage{algorithmicx}
\usepackage{algorithm}
\usepackage{algpseudocode}
\usepackage{color}

\usepackage{empheq} 
\usepackage{fixdif}
\usepackage{physics}
\usepackage{comment}

\begin{document}

\bigskip

\bigskip
\begin{frontmatter}

\title{  
Generalized Probability Density Approach to Histopolation Schemes of Arbitrary Order}

\author[address-Nis]{Gradimir V. Milovanovi{\'c}\corref{corrauthor}}\cortext[corrauthor]{Corresponding author}
\ead{gvm@mi.sanu.ac.rs}

  \author[address-CS]{Federico Nudo}
  
  \ead{federico.nudo@unical.it}

  \address[address-Nis]{ Serbian Academy of Sciences and Arts (SASA), 11000 Belgrade, Serbia \& University of Ni\v s, Faculty of Sciences and Mathematics, 18000 Ni\v s, Serbia}
  \address[address-CS]{Department of Mathematics and Computer Science, University of Calabria, Rende (CS), Italy}

\begin{abstract}
In this paper, we investigate the reconstruction of a bivariate function from weighted edge integrals on a triangular mesh, a problem of central importance in tomography, computer vision, and numerical approximation. Our approach is based on local histopolation methods defined through unisolvent triples, where the edge weights are induced by probability densities. We present a general strategy that applies to arbitrary polynomial order~$k$, in which edge moments are taken against orthogonal polynomials associated with the chosen densities. This yields a systematic framework for weighted reconstructions of any degree, with theoretical guarantees of unisolvency and fully explicit basis functions. As a concrete and flexible instance, we introduce a two-parameter family of Jacobi-type distributions on $[-1,1]$, together with its symmetric Gegenbauer subclass, and show how these densities generate new quadratic reconstruction operators that generalize the standard linear histopolation scheme while preserving its simplicity and locality. We employ an adaptive parameter selection algorithm for Jacobi densities, which 
automatically tunes the distribution parameters to minimize the global reconstruction error. This strategy enhances robustness and adaptivity across different function classes and mesh resolutions. The effectiveness of the proposed operators is demonstrated through extensive numerical experiments, which confirm their superior accuracy in approximating both smooth and highly oscillatory functions. Finally, the framework is sufficiently general to accommodate any admissible edge density, thus providing a flexible and broadly applicable tool for weighted function reconstruction.
\end{abstract}


\end{frontmatter}
\section{Introduction}

The classical framework of polynomial interpolation, traditionally based on pointwise data, 
can be generalized to settings where the available information consists of linear functionals 
acting on the target function~\cite{Rivlin}. 
When these functionals are given by integrals over geometric regions, one speaks of 
\emph{histopolation}~\cite{Robidoux}, and the corresponding approximation procedures are known as 
\emph{histopolation methods}.  

This paradigm provides a natural way to reconstruct functions from averaged data, 
a situation that frequently arises in applied mathematics and engineering. 
In many applications, direct sampling of function values is infeasible, and only integral 
measurements are available. Prominent examples are \emph{computed tomography} and 
\emph{medical imaging}, where devices record line or area integrals of the unknown field~\cite{kak2001principles,natterer2001mathematics,palamodov2016reconstruction}. 
Beyond imaging, histopolation techniques have been applied in 
\emph{computer vision}, \emph{signal and image processing}~\cite{Bosner:2020:AOC}, 
\emph{approximation theory} with curves and splines~\cite{Fischer:2005:MPR,Fischer:2007:CSP,Siewer:2008:HIS,Hallik:2017:QLR,dell2025c}, 
the study of \emph{fractal functions}~\cite{Barnsley:2023:HFF}, and the 
\emph{numerical preservation of physical quantities}~\cite{HiemstraJCP}. 
The last few years have witnessed renewed interest in these ideas~\cite{BruniErbFekete,bruni2024polynomial}, 
leading to global histopolation–regression methods in one and several variables~\cite{bruni2025polynomial,bruno2025bivariate} 
and weighted local polynomial approaches in the bivariate setting~\cite{nudo1,nudo2,dell2025reconstructing,dell2025truncated,milov1}.

A common design principle underlying these methods is the subdivision of the computational 
domain $\Omega$ into smaller elements, where local approximations are built and then combined 
to form a global representation. Depending on whether continuity is enforced across element 
boundaries, the resulting schemes fall into the categories of \emph{conforming} or 
\emph{nonconforming} methods.

A local approximation scheme can be formally described by means of a triple~\cite{Guessab:2022:SAB}
\[
\mathcal{M}_d=\left(S_d,\mathbb{F}_{S_d},\Sigma_{S_d}\right),
\]
where
\begin{itemize}
    \item $S_d$ denotes a polytope in $\mathbb{R}^d$, with $d \geq 1$;
    \item $\mathbb{F}_{S_d}$ is an $n$-dimensional trial space of functions defined on $S_d$;
    \item $\Sigma_{S_d}=\left\{\mathcal{L}_j \, : \,j=1,\dots,n\right\}$ is a set of linearly independent 
    linear functionals, referred to as the \emph{degrees of freedom}.
\end{itemize}
The triple $\mathcal{M}_d$ is called \textit{unisolvent} if the only function in $\mathbb{F}_{S_d}$ annihilated by all the functionals in $\Sigma_{S_d}$ is the zero function. 
Equivalently, there exists a unique basis 
\[
\B{B}=\left\{\varphi_i \,:\, i=1,\dots,n\right\}
\]
of $\mathbb{F}_{S_d}$ such that 
\[
\mathcal{L}_j\left(\varphi_i\right)=\delta_{ij}=
    \begin{cases}
      1, & i=j,\\[4pt]
      0, & i\neq j,
    \end{cases}, \quad i,j=1,\dots,n,
\]
where $\delta_{ij}$ is the Kronecker delta symbol. In this case, the functions $\varphi_i$, $i=1,\dots,n$ constitute the basis associated with 
the triple $\mathcal{M}_d$. When the degrees of freedom are expressed as integrals, the resulting 
approximation scheme is known as a \emph{local histopolation method}.  

A classical example is the nonconforming linear histopolation scheme on triangular meshes, denoted by $\mathcal{CH}$,
where the degrees of freedom are given by edge integrals and the trial space consists of 
piecewise linear functions. This method is computationally appealing and conceptually 
simple, but its limited approximation power prevents it from efficiently reproducing 
oscillations or sharp variations; higher accuracy can only be reached through extensive 
mesh refinement.  

The goal of this paper is to enrich local histopolation schemes by employing polynomial trial
spaces of arbitrary degree $k$, combined with weighted degrees of freedom. The weights are given by
probability densities along the edges, and the corresponding data are expressed as
$\omega$–weighted moments with respect to an orthogonal polynomial basis. This construction yields
linearly independent edge constraints and allows one to derive explicit basis functions through the
inverse of the degrees-of-freedom matrix. As a concrete example, we focus on the two-parameter
family of \emph{Jacobi-type densities}, together with the symmetric Gegenbauer subclass. Although
the framework applies to any $k\ge2$, particular attention is devoted to the quadratic case $k=2$,
where orthogonality leads to a block structure and compact analytic formulas; in the symmetric
(Gegenbauer) setting, the vanishing of odd moments further simplifies the expressions.

The paper is organized as follows. 
Section~\ref{sec1} develops the general local histopolation scheme of arbitrary order. 
In particular, we introduce the $\omega$–weighted edge moments against orthogonal polynomials, 
augment them with interior bubble functionals, prove the unisolvency of the triple $\mathcal{H}_k^\omega$, 
and derive explicit basis functions together with the corresponding reconstruction operator. 
Section~\ref{sec2} is devoted to Jacobi-type edge densities and their symmetric Gegenbauer subclass, 
including normalization and closed-form expressions for the first moments, with special emphasis on 
the quadratic case $k=2$. 
Section~\ref{sec3} addresses the adaptive parameter-selection strategy for the Jacobi family, 
based on mesh-level validation and applicable across different function classes and refinement levels. 
Finally, Section~\ref{sec4} reports numerical experiments that demonstrate the accuracy and robustness 
of the proposed methods.

\section{General local histopolation method of arbitrary order}\label{sec1}

Let $T \subset \mathbb{R}^2$ be a nondegenerate triangle with vertices 
$\B v_1, \B v_2, \B v_3$ and barycentric coordinates $\lambda_1,\lambda_2,\lambda_3$. 
These coordinates satisfy the following fundamental properties
\begin{itemize}
    \item \textbf{Affine property:}
    \begin{equation}\label{prop2}
  \lambda_i\left(t\B{x} + (1-t)\B{y}\right) 
  = t\lambda_i(\B{x}) + (1-t)\lambda_i(\B{y}), 
  \quad \B{x},\B{y}\in T,\quad t\in[0,1],
\end{equation}
\item \textbf{Partition of unity:}
  \begin{equation*}
    \sum_{i=1}^3 \lambda_i(\B{x}) = 1, 
    \quad \B{x}\in T,
  \end{equation*}
  \item \textbf{Kronecker delta property at the vertices:}
  \begin{equation*}
    \lambda_i(\B v_j) = \delta_{ij},
    \quad i,j=1,2,3.
  \end{equation*}
\end{itemize}
For simplicity, we use the notations
\begin{equation*}
    \B{v}_4 = \B{v}_1, \quad  \B{v}_5 = \B{v}_2,
\end{equation*}
as well as 
\begin{equation*}
   \lambda_4 = \lambda_1, \quad \lambda_5 = \lambda_2.  
\end{equation*}
As a consequence, we have 
\begin{equation}\label{propstar}
  \lambda_i(\B{x}) = 0, 
  \quad \B{x}\in e_i=\left[v_{i+1},v_{i+2}\right], 
  \quad i=1,2,3.
\end{equation}
We parametrize any edge $e_j$, $j=1,2,3,$ by
\begin{equation*}
\gamma_j(t) = \frac{1+t}{2}\B v_{j+1} + \frac{1-t}{2}\B v_{j+2}, 
\quad t \in [-1,1].    
\end{equation*}
For an integer $k \geq 2$, we denote by
\begin{equation*}
    \mathbb{P}_k(T)=\left\{ p(\B{x}) = \sum_{i+j \leq k} a_{ij} x^i y^j\, : \, a_{ij} \in \mathbb{R}, \ \B x=(x,y)\in T\right\}
\end{equation*}
the space of bivariate polynomials of total degree $k$. In order to construct our general enriched histopolation scheme, we let 
$\omega \in L^1([-1,1])$ be a probability density on $[-1,1]$, that is 
\begin{equation*}
\omega(t) \geq 0, \quad \int_{-1}^1 \omega(t)dt = 1.    
\end{equation*}
We denote by 
\begin{equation}\label{moment}
    m_{n,\omega}=\int_{-1}^1 t^n\omega(t) dt, \quad n\in\mathbb{N}_0=\mathbb{N}\cup \{0\} 
\end{equation}
and we assume that 
\begin{equation*}
    m_{n,\omega}<\infty, \quad n\in\mathbb{N}_0. 
\end{equation*}
We consider the family of orthogonal polynomials
\begin{equation*}
 \left\{\pi_n\right\}_{n\in\mathbb{N}_0},
\end{equation*}
with respect to the weighted inner product
\begin{equation*}
    \left\langle u,v \right\rangle_\omega := \int_{-1}^1 u(t)v(t)\omega(t)dt.
\end{equation*}
This family satisfies
\begin{equation*}
    \deg\left(\pi_n\right)=n, \quad n\in\mathbb{N}_0
\end{equation*}
and 
\begin{equation*}
\left\langle \pi_n,\pi_m \right\rangle_\omega=K_{nm}\delta_{nm},   \quad K_{nn}>0, \quad n,m\in\mathbb{N}_0.
\end{equation*}
We denote the relative norm as
\begin{equation*}
    \left\lVert \pi_n\right\rVert_{\omega}^2=\left\langle \pi_n,\pi_n\right\rangle_{\omega}.
\end{equation*}
We then define the \emph{edge functionals}
\begin{align}\label{newI}
\mathcal{I}_j^\omega(f) &= \int_{-1}^1 f(\gamma_j(t))\omega(t)dt, \quad j=1,2,3, \\ \label{newL}
\mathcal{L}_j^{\omega,m}(f) &= \int_{-1}^1 \pi_m(t)f(\gamma_j(t))\omega(t)dt, 
\quad j=1,2,3, \quad m=2,\dots,k.
\end{align}
Thus, each edge contributes $k$ linearly independent conditions, 
for a total of $3k$ edge functionals.

\begin{remark}
The particular case $k=2$ with an even weight function $\omega$ 
(symmetric with respect to the origin) has already been analyzed in detail in~\cite{newnew}, 
where the construction of the enriched histopolation scheme was first introduced and studied. 
The present formulation extends that approach both to arbitrary weights and to 
higher polynomial degrees $k>2$.
\end{remark}
In order to guarantee the exact reproduction of all polynomials in 
$\mathbb{P}_k(T)$, the edge functionals introduced above are not sufficient. 
Additional functionals acting in the interior of the triangle are therefore required. 
For this reason, we complete the system by introducing degrees of freedom in the interior of $T$. 
In particular, we consider the \emph{bubble space}
\begin{equation*}
    \mathbb{B}_k = \left
\{ \lambda_1\lambda_2\lambda_3 q \, :\, q \in \mathbb{P}_{k-3}(T) \right\}, \quad \dim\left(\mathbb{B}_k\right)=r_k= \frac{(k-1)(k-2)}{2}.
\end{equation*}
Let $\left\{ g_\ell \right\}_{\ell=1}^{r_k}$ be a basis of $\mathbb{P}_{k-3}(T)$, 
and choose a positive density $W$ on $T$ (for instance, $W = 1/|T|$). For each $\ell=1,\dots,r_k$ we define the interior functional
\begin{equation*}
\mathcal{J}_\ell(f) := \iint_T g_\ell(\B x)f(\B x)W(\B x)d\B x.    
\end{equation*}
We consider the set of all linear functionals 
\begin{equation*}
    \Sigma^\omega_k(T) = \left
\{ \mathcal{I}_j^\omega, \mathcal{L}_j^{\omega,m}\, : \, j=1,2,3,\, m=2,\dots,k \right\}\cup \left\{ \mathcal{J}_\ell \, : \, \ell=1,\dots,r_k \right\}.
\end{equation*}
The associated enriched triple is
\begin{equation*}
\mathcal{H}^\omega_k = \left(T,\mathbb{P}_k(T), \Sigma^\omega_k(T) \right).    
\end{equation*}

\begin{theorem}
For every $k \geq 2$, the triple $\mathcal{H}^\omega_k$ is unisolvent. 
\end{theorem}

\begin{proof}
Let $p \in \mathbb{P}_k(T)$ be a polynomial such that
\begin{eqnarray}\label{1}
    \mathcal{I}^{\omega}_j(p)&=&0, \quad j=1,2,3,\\ \label{2}\mathcal{L}^{\omega,m}_j(p)&=&0, \quad j=1,2,3, \quad m=2,\dots,k,\\ \label{3}  \mathcal{J}_{\ell}(p)&=&0, \quad  \ell=1,\dots,r_k.
\end{eqnarray}
Since $p \in \mathbb{P}_k(T)$, its restriction to any edge $e_j$ is a univariate polynomial of degree at most $k$, i.e.
\begin{equation*}
p\left(\gamma_j(t)\right)\in\mathbb{P}_k([-1,1]), \quad j=1,2,3. 
\end{equation*}
Hence we may expand
\begin{equation}\label{asu1}
p\left(\gamma_j(t)\right)=\sum_{l=0}^ka_{l,j} \pi_l(t). 
\end{equation}
Multiplying both sides by $\pi_m(t)\omega(t)$ and integrating over $[-1,1]$, orthogonality gives
\begin{equation*}
\mathcal{L}_j^{\omega,k}(p)=\int_{-1}^1 \pi_k(t) p\left(\gamma_j(t)\right)\omega(t)dt=\sum_{l=0}^ka_{l,j} \int_{-1}^1 \pi_k(t)\pi_l(t)\omega(t)dt=a_{k,j} \left\lVert\pi_k \right\rVert^2_{\omega}.
\end{equation*}
By condition~\eqref{2}, we have $\mathcal{L}_j^{\omega,k}(p)=0$, which implies $a_{k,j}=0$. Repeating the same argument for $m=k,k-1,\dots,2$, we have
\begin{equation*} p\left(\gamma_j(t)\right)=a_{0,j}\pi_0(t)+a_{1,j}\pi_1(t).
\end{equation*}
Then the restriction of $p$ to each edge $e_j$ is affine. 
It is a classical fact that any polynomial of degree at most $k$ whose restriction 
to every edge of $T$ is affine necessarily belongs to
\begin{equation*}
\mathbb{P}_1(T) \oplus \mathbb{B}_k.  
\end{equation*}
 Hence we can write
\begin{equation*}
p = p_1 + q, \quad p_1 \in \mathbb{P}_1(T),\quad q \in \mathbb{B}_k.    
\end{equation*}
Applying condition~\eqref{1} to $p$, and using the one-point Gaussian quadrature rule with respect to the weight $\omega$, we obtain
\begin{equation*}
0=\mathcal{I}_j^{\omega}(p)=\mathcal{I}_j^{\omega}\left(p_1\right)= w_j p_1\left(\gamma_j(t_j)\right), \quad t_j\in (-1,1), \quad w_j>0, \quad j=1,2,3.
\end{equation*}
Therefore $p_1$ is a linear polynomial vanishing at three noncollinear points, which forces $p_1=0$.
Thus $p$ must belong to the bubble space, and can be written as
\begin{equation*}
    p(\B x)=\sum_{\kappa=1}^{r_k} b_{\kappa} \lambda_1\lambda_2\lambda_3 g_{\kappa},
\end{equation*}
where $b_{\kappa}\in\mathbb{R}$, $g_{\kappa}\in\mathbb{P}_{k-3}(T)$, $\kappa=1,\dots,r_k$. Equivalently, we can write
\begin{equation*}
    p= \lambda_1\lambda_2\lambda_3 q, \quad q=\sum_{\kappa=1}^{r_k}b_{\kappa} g_{\kappa}\in\mathbb{P}_{k-3}(T). 
\end{equation*}
Then, conditions~\eqref{3} become
\begin{equation*}
J_\ell(p)=J_\ell\left(\lambda_1\lambda_2\lambda_3 q\right) = \iint_T g_\ell(\B x)\lambda_1(\B x)\lambda_2(\B x)\lambda_3(\B x)
q(\B x)W(\B x)d \B x = 0, \quad \ell=1,\dots,r_k.    
\end{equation*}
Since the weight $\lambda_1\lambda_2\lambda_3W$ is strictly positive in the 
interior of $T$, the bilinear form
\begin{equation*}
(q,g) \mapsto \iint_T q(\B x)g(\B x)\lambda_1(\B x)\lambda_2(\B x)\lambda_3(\B x)
W(\B x)d\B x    
\end{equation*}
is positive definite on $\mathbb{P}_{k-3}(T)$.  
Hence 
\begin{equation*}
    q=\sum_{\kappa=1}^{r_k}b_{\kappa} g_{\kappa}=0
\end{equation*}
then
\begin{equation*}
    b_{1}=\cdots=b_{r_k}=0,
\end{equation*}
which implies $p=0$. This concludes the proof.
\end{proof}

\begin{remark}
The counting of degrees of freedom is consistent. Indeed, 
\begin{equation*}
\#\Sigma^\omega_k(T) = 3k + \frac{(k-1)(k-2)}{2}
= \frac{(k+1)(k+2)}{2} = \dim\left(\mathbb{P}_k(T)\right)=:K,   
\end{equation*}
where $\#(\cdot)$ denotes the cardinality operator. 
\end{remark}
In this setting, we define the following linear map
\begin{equation}\label{mapHgeneral}
H:q \in \mathbb{P}_k(T)\longrightarrow
H(q) := \left(\mathcal{D}_1^{\mathrm{enr}}(q), \dots, \mathcal{D}_K^{\mathrm{enr}}(q)\right)^{\top} \in \mathbb{R}^K, 
\end{equation}
where 
\begin{equation*}
\left\{\mathcal{D}_r^{\mathrm{enr}}\, :\, r=1,\dots,K\right\}    
\end{equation*}
denote the linearly independent degrees of freedom 
associated with the enriched triple
$\mathcal{H}_k^{\omega}$. 

We represent the linear map $H$ in matrix form with respect to the basis
\begin{equation*}
    \B{B}_{K} = \{\phi_{1}, \dots, \phi_{K}\}
\end{equation*}
of $\mathbb{P}_k(T)$ and the canonical basis of $\mathbb{R}^{K}$,
\begin{equation*}
    \{\boldsymbol{u}_1, \dots, \boldsymbol{u}_{K}\}.
\end{equation*}
Specifically, we define
\begin{equation}\label{matrixHgeneral}
      \hat{H} = \left[\boldsymbol{h}_{1},\dots,\boldsymbol{h}_{K}\right]\in\mathbb{R}^{K\times K},
\end{equation}
where 
\begin{equation*}
    \boldsymbol{h}_{j}= H(\phi_{j}), \quad j=1,\dots,K.
\end{equation*}
We now provide an explicit expression for the basis functions
\begin{equation*}
\mathbf{B}^{\mathrm{enr}}=\left\{\Psi_{\ell} \, :\, \ell=1,\dots,K\right\}
\end{equation*}
for $\mathbb{P}_k(T)$ such that
\begin{equation}
\label{condbasis}  
H\left(\Psi_{\ell}\right)=\boldsymbol{u}_{\ell}, \quad \ell=1,\dots,K.
\end{equation}
These functions are usually referred to as the \textit{basis functions of the triple}  $\mathcal{H}_k^{\omega}$.

\begin{theorem}\label{theo:generalPk}
    The basis functions of the enriched triple $\mathcal{H}_k^{\omega}$ can be written as
    \begin{equation}\label{genbasisgeneral}
        \Psi_{\ell}(\B{x}) = \left\langle \tilde{\boldsymbol{h}}_{\ell},  \boldsymbol{\phi}(\B{x}) \right\rangle,
        \quad \ell=1,\dots,K,
    \end{equation}
    where 
    \begin{equation}\label{psigeneral}
        \boldsymbol{\phi}(\B{x}) = 
        \left(\phi_{1}(\B{x}), \dots, \phi_{K}(\B{x})\right)^{\top},
    \end{equation}
and $\left\{\tilde{\boldsymbol{h}}_{\ell}\right\}_{\ell=1}^K$ denote the column vectors of the inverse matrix 
    $\hat{H}^{-1}$.
\end{theorem}
\begin{proof}
Since $\Psi_{\ell}\in\mathbb{P}_k(T)$, it can be expressed as
\begin{equation}\label{phi1old}
\Psi_{\ell}=\left\langle \boldsymbol{c}^{(\ell)},\boldsymbol{\phi}\right\rangle=\sum_{i=1}^K c^{(\ell)}_i\phi_i, 
\end{equation}
where
\begin{equation*}
\boldsymbol{c}^{(\ell)}=\left(c^{(\ell)}_1,\dots,c^{(\ell)}_K\right)^\top\in\mathbb{R}^{K} 
\end{equation*}
and $\boldsymbol{\phi}$ is defined in~\eqref{psigeneral}.  
Applying $H$ to~\eqref{phi1old} and using linearity gives
\begin{equation*}
H\left(\Psi_{\ell}\right)=\sum_{i=1}^K c_{i}^{(\ell)} H\left(\phi_i\right)= \hat{H}\boldsymbol{c}^{(\ell)}.
\end{equation*}
Imposing condition~\eqref{condbasis} yields
\begin{equation*}
\boldsymbol{u}_{\ell}=\hat{H}\boldsymbol{c}^{(\ell)},
\end{equation*}
and therefore
\begin{equation}\label{ssc}
\boldsymbol{c}^{(\ell)}=\tilde{\boldsymbol{h}}_{\ell}.
\end{equation}
Substituting~\eqref{ssc} into~\eqref{phi1old} gives the representation~\eqref{genbasisgeneral}, 
which completes the proof.
\end{proof}

We now have all the necessary ingredients to define the reconstruction operator associated to the triple $\mathcal{H}^{\omega}_k$.  
This operator is constructed using the basis functions of Theorem~\ref{theo:generalPk} and is defined by
\begin{equation}\label{pi1}
\begin{array}{rcl}
\Pi_{k}^{\omega}\colon C(T) &\longrightarrow& \mathbb{P}_{k}(T) \\[1ex]
f &\longmapsto& \displaystyle\sum_{j=1}^{K} 
   \mathcal{D}^{\mathrm{enr}}_{j}(f)\Psi_{j}.
\end{array}
\end{equation}

In the following, we derive an explicit representation of the basis functions 
in the quadratic case $k=2$. 
To this end, we first establish a few auxiliary lemmas that will be used 
in the subsequent construction.
\begin{lemma}\label{lemma1}
   For any $i,j=1,2,3,$ it results
   \begin{equation*}
       \mathcal{I}_{j}^{\omega}\left(\lambda_i\right) =
\begin{cases}
0, & \text{if } i=j,\\[6pt]
\dfrac{1+m_{1,\omega}}{2}, & \text{if } i=j+1, \\[10pt]
\dfrac{1-m_{1,\omega}}{2}, & \text{if } i=j+2,
\end{cases}
   \end{equation*}
where $m_{1,\omega}$ is the first moment of the weight function $\omega$. 
\end{lemma}
\begin{proof}
We prove the statement for $j=1$; the cases $j=2,3$ follow by cyclic permutation of the indices. 
Since, by~\eqref{propstar}, along $e_1$ we have $\lambda_1=0$, it is immediate that
\begin{equation*}
\mathcal{I}_{1}^{\omega}\left(\lambda_1\right)=0. 
\end{equation*}
On the other hand, by the affine property and the Kronecker
delta property of the barycentric coordinates, we have
\begin{eqnarray*}
\lambda_2\left(\gamma_1(t)\right)
&=&\frac{1+t}{2}\lambda_2(\B{v}_2)+\frac{1-t}{2}\lambda_2(\B{v}_3)
=\frac{1+t}{2},
\\
\lambda_3\left(\gamma_1(t)\right)
&=&\frac{1+t}{2}\lambda_3(\B{v}_2)+\frac{1-t}{2}\lambda_3(\B{v}_3)
=\frac{1-t}{2}.    
\end{eqnarray*}
Hence, using that $\omega$ is a probability density, we get
\begin{eqnarray*}
    \mathcal{I}_1^{\omega}\left(\lambda_2\right)&=&\int_{-1}^1 \frac{1+t}{2}\omega(t)dt
=\frac{1}{2}\int_{-1}^1 \omega(t)dt+\frac{1}{2}\int_{-1}^1 t\omega(t)dt
=\frac{1+m_{1,\omega}}{2},\\
\mathcal{I}_1^{\omega}\left(\lambda_3\right)&=&\int_{-1}^1 \frac{1-t}{2}\omega(t)dt
=\frac{1}{2}\int_{-1}^1 \omega(t)dt-\frac{1}{2}\int_{-1}^1 t\omega(t)dt
=\frac{1-m_{1,\omega}}{2}.
\end{eqnarray*}
The cases $j=2$ and $j=3$ are obtained analogously by cyclic permutation of the indices.
\end{proof}

\begin{lemma}\label{lemma2}
For any $i,k,j\in\{1,2,3\}$ it holds
\begin{equation*}
\mathcal{I}_{j}^{\omega}(\lambda_i\lambda_k)=
\begin{cases}
0, & \text{if } i=j \text{ or } k=j,\\[6pt]
\dfrac{1-m_{2,\omega}}{4}, & \text{if } i=j+1 \text{ and } k=j+2,\\[10pt]
\dfrac{1+2m_{1,\omega}+m_{2,\omega}}{4}, & \text{if } i=k=j+1,\\[10pt]
\dfrac{1-2m_{1,\omega}+m_{2,\omega}}{4}, & \text{if } i=k=j+2.
\end{cases}
\end{equation*}
\end{lemma}

\begin{proof}
As in the previous lemma, we prove the statement for $j=1$; the cases $j=2,3$ follow by cyclic permutation of the indices.
Along $e_1$ we have $\lambda_1=0$ and
\begin{equation*}
\lambda_2\left(\gamma_1(t)\right)=\frac{1+t}{2}, \quad \lambda_3\left(\gamma_1(t)\right)=\frac{1-t}{2}.
\end{equation*}
Hence, if $i=1$ or $k=1$, then $\lambda_i\lambda_k=0$ on $e_1$ and
\begin{equation*}
\mathcal{I}_1^{\omega}\left(\lambda_i\lambda_k\right)=0.    
\end{equation*}
If $i=2$ and $k=3$, then
\begin{equation*}
\lambda_2\left(\gamma_1(t)\right)\lambda_3\left(\gamma_1(t)\right)=\frac{1-t^2}{4},    
\end{equation*}
and, using that $\omega$ is a probability density, we have
\begin{equation*}
    \mathcal{I}_1^{\omega}\left(\lambda_2\lambda_3\right)=\int_{-1}^1 \frac{1-t^2}{4}\omega(t)dt
=\frac{1}{4}\left(\int_{-1}^1 \omega(t)dt-\int_{-1}^1 t^2\omega(t) dt\right)
=\frac{1-m_{2,\omega}}{4}.
\end{equation*}
If $i=k=2$, then
\begin{equation*}
\lambda_2^2\left(\gamma_1(t)\right)=\frac{(1+t)^2}{4}
=\frac{1+2t+t^2}{4},
\end{equation*}
so
\begin{equation*}
\mathcal{I}_1^{\omega}\left(\lambda_2^2\right)=\int_{-1}^1 \frac{1+2t+t^2}{4}\omega(t)dt
=\frac{1+2m_{1,\omega}+m_{2,\omega}}{4}.    
\end{equation*}
Similarly, if $i=k=3$,
\begin{equation*}
\lambda_3^2\left(\gamma_1(t)\right)=\frac{(1-t)^2}{4}
=\frac{1-2t+t^2}{4}
\end{equation*}
and
\begin{equation*}
    \mathcal{I}_1^{\omega}\left(\lambda_3^2\right)=\int_{-1}^1 \frac{1-2t+t^2}{4}\omega(t)dt
=\frac{1-2m_{1,\omega}+m_{2,\omega}}{4}.
\end{equation*}
This completes the proof; the cases $j=2,3$ follow analogously by cyclic permutation of the indices.
\end{proof}

\begin{lemma}\label{lemma3}
   For any $i,j=1,2,3,$ it results
   \begin{equation*}
\mathcal{L}_{j}^{\omega,2}\left(\lambda_i\right) =0.
   \end{equation*}
\end{lemma}
\begin{proof}
It follows directly from the orthogonality of $\pi_2$.
\end{proof}

\begin{lemma}\label{lemma4}
For any $i,k,j=1,2,3$ it holds
\begin{equation*}
\mathcal{L}_{j}^{\omega,2}\left(\lambda_i\lambda_k\right)
= -\frac{\rho_2}{4}\left(1-\delta_{ij}\right)\left(1-\delta_{kj}\right),
\end{equation*}
where 
\begin{equation}\label{M2}
   \rho_2=\left\langle \pi_2, t^2\right\rangle_{\omega}.
\end{equation}
\end{lemma}

\begin{proof}
We fix $j\in \{1,2,3\}$. From~\eqref{propstar}, along $e_j$ we have $\lambda_j=0$, while
\begin{equation*}
\lambda_{j+1}\left(\gamma_j(t)\right)=\frac{1+t}{2}, \quad  \lambda_{j+2}\left(\gamma_j(t)\right)=\frac{1-t}{2}.
\end{equation*}
Then, if $i=j$ or $k=j$, we get
\begin{equation*}
    \lambda_i\lambda_k=0
\end{equation*}
on $e_j$ and hence
$$\mathcal{L}_j^{\omega,2}\left(\lambda_i\lambda_k\right)=0.$$ 
Otherwise if 
\begin{equation*}
    (i,k)=(j+1,j+2),
\end{equation*}
then, by using the affine property of the barycentric coordinates, we have
\begin{eqnarray*}
\mathcal{L}_{j}^{\omega,2}\left(\lambda_{j+1}\lambda_{j+2}\right)
&=&\int_{-1}^1 \pi_2(t)\left(\lambda_{j+1}\lambda_{j+2}\right)\left(\gamma_j(t)\right)\omega(t)dt\\
&=&\frac{1}{4}\int_{-1}^1 \pi_2(t) \frac{1-t^2}{4} \omega(t) dt\\&=&
\frac{1}{4}\left(\langle \pi_2,1\rangle_{\omega}-\langle \pi_2,t^2\rangle_{\omega}\right)
=-\frac{\rho_2}{4},
\end{eqnarray*}
where in the last equation we have used the orthogonality property of the polynomial $\pi_2$. This gives the claimed formula in all cases.
\end{proof}

\begin{remark}
    We observe that $\rho_2$ defined in~\eqref{M2} can be expressed in terms of the moments $m_{k,\omega}$. Indeed, assuming to work with a family of monic orthogonal polynomials, we can write
    \begin{equation*}
\pi_2(t)=t^2-at-b, \quad a,b\in\mathbb{R}.  
    \end{equation*}
The orthogonality conditions
\begin{equation*}
\left\langle \pi_2,1\right\rangle_\omega=0,\quad \left\langle \pi_2,t\right\rangle_\omega=0  
\end{equation*}
give the linear system
\begin{equation*}
    m_{2,\omega}-a m_{1,\omega}-b=0,\quad m_{3,\omega}-a m_{2,\omega}-bm_{1,\omega}=0.
\end{equation*}
Solving for $a$ and $b$ yields
\begin{equation*}
a=\frac{m_{3,\omega}-m_{1,\omega}m_{2,\omega}}{m_{2,\omega}-m_{1,\omega}^2},\quad
b=m_{2,\omega}-a m_{1,\omega}.    
\end{equation*}
Then
\begin{equation*}
    \rho_2=\left\langle \pi_2,t^2\right\rangle_\omega
=\int_{-1}^1 \left(t^2-at-b\right)t^2\omega(t)dt
= m_{4,\omega}-am_{3,\omega}-bm_{2,\omega}.
\end{equation*}
Substituting $b=m_{2,\omega}-a m_{1,\omega}$ gives
\begin{equation*}
    \rho_2
= m_{4,\omega}-am_{3,\omega}-m_{2,\omega}^2+am_{1,\omega}m_{2,\omega}
= m_{4,\omega}-m_{2,\omega}^2 - a\left(m_{3,\omega}-m_{1,\omega}m_{2,\omega}\right).
\end{equation*}
Using $a=\dfrac{m_{3,\omega}-m_{1,\omega}m_{2,\omega}}{m_{2,\omega}-m_{1,\omega}^2}$ we finally obtain the compact formula
\begin{equation}\label{rhomoment}
\rho_2= m_{4,\omega} - m_{2,\omega}^2 - \frac{\left(m_{3,\omega}-m_{1,\omega}m_{2,\omega}\right)^2}{m_{2,\omega}-m_{1,\omega}^2} .    
\end{equation}
\end{remark}

Now we are ready to compute the analytic expression of the basis functions relative to $\mathcal{H}_2^{\omega}$. To this aim, we consider the following map 
\begin{equation*}
H:\mathbb{P}_2(T)\to\mathbb{R}^6,\quad
H(q)=\left(\mathcal{I}_1^\omega(q), \mathcal{I}_2^\omega(q), \mathcal{I}_3^\omega(q),  \mathcal{L}_1^{\omega,2}(q),
\mathcal{L}_2^{\omega,2}(q),
\mathcal{L}_3^{\omega,2}(q)\right)^\top.    
\end{equation*}
Let $\hat{H}\in\mathbb{R}^{6\times 6}$ be the matrix of $H$ expressed relative to the basis
\begin{equation*}
\B{B}_6=\left\{\lambda_1, \lambda_2, \lambda_3, \lambda_1\lambda_2, \lambda_3\lambda_1, \lambda_2\lambda_3\right\}    
\end{equation*}
for $\mathbb{P}_2(T)$. In this case, using Lemmas~\ref{lemma1},~\ref{lemma2},~\ref{lemma3} and~\ref{lemma4}, we have
\begin{equation*}
\hat{H}=
\begin{pmatrix}
0 & \dfrac{1+m_{1,\omega}}{2} & \dfrac{1-m_{1,\omega}}{2} & 0 & 0 & \dfrac{1-m_{2,\omega}}{4}\\[6pt]
\dfrac{1-m_{1,\omega}}{2} & 0 & \dfrac{1+m_{1,\omega}}{2} & 0 & \dfrac{1-m_{2,\omega}}{4} & 0\\[6pt]
\dfrac{1+m_{1,\omega}}{2} & \dfrac{1-m_{1,\omega}}{2} & 0 & \dfrac{1-m_{2,\omega}}{4} & 0 & 0\\[6pt]
0 & 0 & 0 & 0 & 0 & -\dfrac{\rho_2}{4}\\[6pt]
0 & 0 & 0 & 0 & -\dfrac{\rho_2}{4} & 0\\[6pt]
0 & 0 & 0 & -\dfrac{\rho_2}{4} & 0 & 0
\end{pmatrix}.
\end{equation*}
Its determinant is
\begin{equation*}
\det\left(\hat{H}\right)=\frac{\rho_2^{3}}{256}\left(3m_{1,\omega}^2+1\right)\neq 0   
\end{equation*}
since $\rho_2$, defined in~\eqref{M2}, is nonzero. We now introduce the shorthand notation

\begin{align}
\sigma_0 &:= \frac{m_{1,\omega}^2-1}{3m_{1,\omega}^2+1}, &
\sigma_+ &:= \frac{(m_{1,\omega}+1)^2}{3m_{1,\omega}^2+1}, &
\sigma_- &:= \frac{(m_{1,\omega}-1)^2}{3m_{1,\omega}^2+1}, \label{sigmas} \\[0.4em]
\tau &:= \frac{1-m_{2,\omega}}{\rho_2\left(3m_{1,\omega}^2+1\right)}, &
\nu_0 &:= \tau(m_{1,\omega}^2-1), &
\nu_+ &:= \tau(m_{1,\omega}+1)^2, &
\nu_- &:= \tau(m_{1,\omega}-1)^2. \label{tausnus}
\end{align}
With this notation, the inverse of $\hat{H}$ takes the form
\begin{equation}\label{invmat}
\widehat H^{-1} =
\begin{pmatrix}
\sigma_0 & \sigma_- & \sigma_+ & \nu_0 & \nu_{-} & \nu_{+} \\
\sigma_+ & \sigma_0 & \sigma_- &  \nu_{+} &  \nu_{0} &  \nu_{-} \\
\sigma_- & \sigma_+ & \sigma_0 &  \nu_{-} &  \nu_{+} &  \nu_{0} \\
0 & 0 & 0 & 0 & 0 & -\dfrac{4}{\rho_2} \\
0 & 0 & 0 & 0 & -\dfrac{4}{\rho_2} & 0 \\
0 & 0 & 0 & -\dfrac{4}{\rho_2} & 0 & 0 \\
\end{pmatrix}.    
\end{equation}
We next derive closed-form expressions for the basis functions associated with the enriched triple $\mathcal{H}_{2}^{\omega}$.  
Specifically, we determine a basis
\begin{equation*}
   \mathbf{B}_{2,\omega}=\left\{\varphi_{i}, \psi_{i} \, :\, i=1,2,3\right\}
\end{equation*}
of the vector space $\mathbb{P}_{2}(T)$ satisfying the conditions
\begin{align*}
{{\mathcal{I}}}_{j}^{\omega}\left(  \varphi_{i}\right) &=  \delta_{ij},
  \\ 
{{\mathcal{L}}}^{\omega,2}_{j}\left(\varphi_{i}\right) &=  0,  \\ 
{{\mathcal{I}}}_j^{\omega}\left(\psi_{i}\right) &= 0, \\ {{\mathcal{L}}}^{\omega,2}_{j}\left(\psi_{i}\right) &=  \delta_{ij}, 
\end{align*}
for any $i,j=1,2,3$. 
\begin{theorem}\label{newth}
 The basis functions of the enriched triple $\mathcal{H}_2^{\omega}$ can be written as
 \begin{eqnarray*}
\varphi_j&=&\sigma_0\lambda_j+\sigma_+\lambda_{j+1}+\sigma_{-}\lambda_{j+2}, \quad j=1,2,3,\\
\psi_j&=& \nu_0\lambda_j+\nu_+\lambda_{j+1}+\nu_{-}\lambda_{j+2}-\frac{4}{\rho_2}\lambda_{j+1}\lambda_{j+2}, \quad j=1,2,3, 
 \end{eqnarray*}
 where $\sigma_0,\sigma_{+},\sigma_{-},\tau,\nu_0,\nu_{+},\nu_{-}$ are defined in~\eqref{sigmas} and~\eqref{tausnus}.
\end{theorem}
\begin{proof}
 The result follows directly from Theorem~\ref{theo:generalPk} and the inverse matrix~\eqref{invmat}.
\end{proof}

\begin{remark}
If the probability density $\omega$ is symmetric with respect to the origin, then its odd moments vanish, in particular 
$m_{1,\omega}=0$. In this case the coefficients defined above simplify to
\begin{equation*}
\sigma_0 = -1, 
\quad \sigma_+ = \sigma_- = 1, 
\quad \nu_0 = -\tau, 
\quad \nu_+ = \nu_- = \tau,    
\end{equation*}
where $$\tau = \dfrac{1-m_{2,\omega}}{\rho_2}.$$
Consequently, the expressions of the basis functions $\varphi_j,\psi_j,$ $j=1,2,3$ given in Theorem~\ref{newth} 
become significantly simpler 
\begin{equation*}
    \varphi_j = 1-2\lambda_j, 
\quad 
\psi_j = \tau \varphi_j -\frac{4}{\rho_2}\lambda_{j+1}\lambda_{j+2}.  
\end{equation*}
Moreover, the expression of $\rho_2$, given in~\eqref{rhomoment}, simplifies to
\begin{equation*}
    \rho_2 = m_{4,\omega} - m_{2,\omega}^2.
\end{equation*}    
\end{remark}

\section{Jacobi-type probability density}\label{sec2}
In this section, we introduce a two-parameter family of admissible edge densities
provided by the Jacobi probability distribution on $[-1,1]$. Specifically, we set
\begin{equation}\label{jacobi-pdf}
\omega_{\alpha,\beta}(t) = a_{\alpha,\beta}(1-t)^{\alpha}(1+t)^{\beta},
\quad t \in [-1,1], \quad \alpha,\beta>-1,
\end{equation}
where $a_{\alpha,\beta}$ is a normalization constant. To study this case, we denote by 
\begin{equation*}
B(x,y)=\int_{0}^{1} u^{x-1}(1-u)^{y-1}du, \quad x,y>0,    
\end{equation*}
the classical Euler Beta function~\cite{Abramowitz:1948:HOM}. This function satisfies, among others, the following properties:
\begin{itemize}
    \item \textbf{Symmetry:} $B\left(x,y\right)=B\left(y,x\right)$ for all $x,y>0$,\\
    \item \textbf{Recurrence relation:} $B(x+1,y)=\dfrac{x}{x+y}B(x,y)$ for all $x,y>0$.
\end{itemize}

\begin{lemma}\label{lem:jacobi-normalization}
The function $\omega_{\alpha,\beta}$ is a probability density function on $[-1,1]$ 
if and only if
\begin{equation*}
    a_{\alpha,\beta}=\frac{1}{2^{\alpha+\beta+1}B(\alpha+1,\beta+1)}.
\end{equation*}
\end{lemma}

\begin{proof}
To determine the normalization constant, we 
perform the change of variables $$u=\frac{1+t}{2}.$$ Then, we have 
\begin{eqnarray*}
\int_{-1}^{1}\omega_{\alpha,\beta}(t)dt&=&
a_{\alpha,\beta}\int_{-1}^{1}(1-t)^{\alpha}(1+t)^{\beta}dt\\
&=&  2^{\alpha+\beta+1} a_{\alpha,\beta}\int_{0}^{1} u^{\beta}(1-u)^{\alpha}du
=  2^{\alpha+\beta+1} a_{\alpha,\beta} B(\beta+1,\alpha+1).   
\end{eqnarray*}
By the symmetry property $B(\alpha,\beta)=B(\beta,\alpha)$, 
the normalization condition holds if and only if
\begin{equation*}
    a_{\alpha,\beta}=\frac{1}{2^{\alpha+\beta+1}B(\alpha+1,\beta+1)},
\end{equation*}
which proves the claim.
\end{proof}

Therefore, in the following we shall work with the normalized Jacobi density
\begin{equation}\label{jacobi-normalized}
\omega_{\alpha,\beta}(t)=\frac{(1-t)^{\alpha}(1+t)^{\beta}}
     {2^{\alpha+\beta+1}B(\alpha+1,\beta+1)}.
\end{equation}
    \begin{figure}
      \centering
\includegraphics[width=0.49\linewidth]{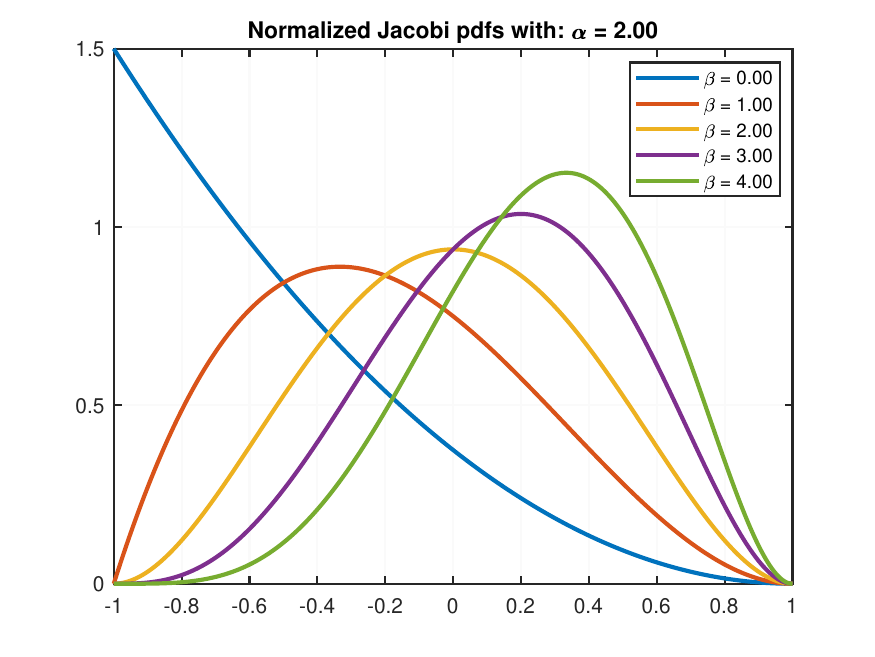}
\includegraphics[width=0.49\linewidth]{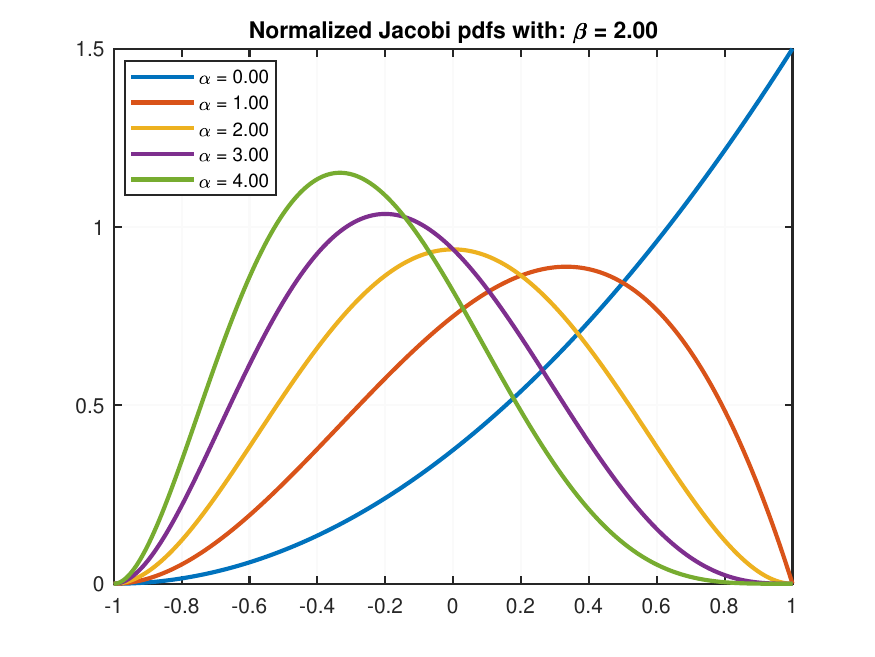}
      \caption{Normalized Jacobi probability density functions $\omega_{\alpha,\beta}$ on $[-1,1]$.
Left: effect of the parameter $\beta$, with fixed $\alpha = 2$ and different values of $\beta$.
Right: effect of the parameter $\alpha$, with fixed $\beta = 2$ and varying values of $\alpha$, illustrating how the weight near the endpoints $t=\pm 1$ can be shifted and modulated.
These plots highlight the flexibility of Jacobi densities in shaping local features through the choice of $\alpha$ and $\beta$.}
      \label{fig:densities}
  \end{figure}
As shown in Figure~\ref{fig:densities}, increasing $\alpha$ or $\beta$ shifts the density toward the endpoints $t=-1$ and $t=+1$, respectively, demonstrating the local shaping effect of the Jacobi parameters. To simplify the notation, in the following we write $\mathcal{H}_2^{\alpha,\beta}$.

In order to characterize the action of the Jacobi density within our histopolation framework,
we compute its moments. The first two moments admit simple closed forms,
while higher-order moments can be expressed in terms of Beta functions.

\begin{lemma}\label{lem:jacobi-moments}
Let $\omega_{\alpha,\beta}$ be the normalized Jacobi density defined in~\eqref{jacobi-normalized}. Then its first and second moments are given by
\begin{align*}
m_{1,\omega} &= \int_{-1}^1 t\omega_{\alpha,\beta}(t)dt 
= \frac{\beta-\alpha}{\alpha+\beta+2},\\[0.5em]
m_{2,\omega} &= \int_{-1}^1 t^2\omega_{\alpha,\beta}(t)dt 
= \frac{(\alpha-\beta)^2+(\alpha+\beta+2)}{(\alpha+\beta+2)(\alpha+\beta+3)}.
\end{align*}
\end{lemma}

\begin{proof}
Using the change of variables
\begin{equation*}
u=\frac{1+t}{2},    
\end{equation*}
we obtain
\begin{align*}
m_{1,\omega} &= \int_{-1}^1 t\omega_{\alpha,\beta}(t)dt 
= \frac{1}{2^{\alpha+\beta+1}B(\alpha+1,\beta+1)} 
\int_{-1}^1 t(1-t)^{\alpha}(1+t)^{\beta}dt \\
&= \frac{1}{B(\alpha+1,\beta+1)}
\int_0^1 (2u-1)u^{\beta}(1-u)^{\alpha}du \\
&= \frac{1}{B(\alpha+1,\beta+1)}\left[2B(\alpha+1,\beta+2)-B(\alpha+1,\beta+1)\right].
\end{align*}
Using the recurrence relation of the beta function, we have
\begin{equation*}
    m_{1,\omega} = \frac{2(\beta+1)}{\alpha+\beta+2}-1=\frac{\beta-\alpha}{\alpha+\beta+2}.
\end{equation*}
We can proceed analogously for the second moment and obtain
\begin{align*}
m_{2,\omega} &= \frac{1}{B(\alpha+1,\beta+1)}
\int_0^1 (2u-1)^2u^{\beta}(1-u)^{\alpha}du \\
&= \frac{1}{B(\alpha+1,\beta+1)}\left[
4B(\beta+3,\alpha+1)-4B(\beta+2,\alpha+1)+B(\beta+1,\alpha+1)\right].
\end{align*}
Applying recurrence formulas for the Beta function, by easy calculations, we get
\begin{equation*}
    m_{2,\omega}=\frac{(\alpha-\beta)^2+(\alpha+\beta+2)}{(\alpha+\beta+2)(\alpha+\beta+3)}.
\end{equation*}
This completes the proof.
\end{proof}

We now compute the higher-order moments needed in the definition of $\rho_2$. 
In particular, the third and fourth moments of the Jacobi density admit closed forms.
\begin{lemma}\label{lem:jacobi-m3m4}
Let $\omega_{\alpha,\beta}$ be the normalized Jacobi density defined in~\eqref{jacobi-normalized}. Then
\begin{align*}
m_{3,\omega} &=\frac{(\beta-\alpha)\left[(\alpha-\beta)^2 + 3(\alpha+\beta) + 8\right]}
{(\alpha+\beta+2)(\alpha+\beta+3)(\alpha+\beta+4)},\\[0.6em]
m_{4,\omega} &= \frac{(\alpha-\beta)^4 + 6(\alpha-\beta)^2(\alpha+\beta) + 20(\alpha-\beta)^2 + 3(\alpha+\beta)^2 + 18(\alpha+\beta) + 24}
{(\alpha+\beta+2)(\alpha+\beta+3)(\alpha+\beta+4)(\alpha+\beta+5)}.
\end{align*}
\end{lemma}

\begin{proof}
Using the change of variables $u=\tfrac{1+t}{2}$, we have
\begin{equation*}
    m_{k,\omega}=\int_{-1}^1 t^k\omega_{\alpha,\beta}(t)dt
= \frac{1}{B(\alpha+1,\beta+1)}\int_0^1 (2u-1)^ku^{\beta}(1-u)^{\alpha}du.
\end{equation*}
Expanding $(2u-1)^3=8u^3-12u^2+6u-1$ and integrating term by term gives
\begin{align*}
m_{3,\omega}
&= \frac{8B(\beta+3,\alpha+1)-12B(\beta+2,\alpha+1)+6B(\beta+1,\alpha+1)-B(\beta,\alpha+1)}{B(\alpha+1,\beta+1)}.
\end{align*}
Applying the recurrence relation of the Beta function and simplifying, we get
\begin{equation*}
m_{3,\omega}
= \frac{(\beta-\alpha)\left[(\alpha-\beta)^2 + 3(\alpha+\beta) + 8\right]}
{(\alpha+\beta+2)(\alpha+\beta+3)(\alpha+\beta+4)}.  
\end{equation*}
Similarly, using the expansion $(2u-1)^4=16u^4-32u^3+24u^2-8u+1$, we find
\begin{align*}
m_{4,\omega}
&= \frac{1}{B(\alpha+1,\beta+1)}\!\left[
16B(\beta+4,\alpha+1)-32B(\beta+3,\alpha+1)+24B(\beta+2,\alpha+1)\right.\\
&\hspace{8.3em}\left.-8B(\beta+1,\alpha+1)+B(\beta,\alpha+1)
\right].
\end{align*}
Again, applying the recurrence relation leads to
\begin{equation*}
    m_{4,\omega}
= \frac{(\alpha-\beta)^4 + 6(\alpha-\beta)^2(\alpha+\beta) + 20(\alpha-\beta)^2 + 3(\alpha+\beta)^2 + 18(\alpha+\beta) + 24}
{(\alpha+\beta+2)(\alpha+\beta+3)(\alpha+\beta+4)(\alpha+\beta+5)}.
\end{equation*}
This completes the proof.
\end{proof}

\begin{remark}
For the Jacobi-type probability density it is also possible to derive an explicit closed form
for the quantity $\rho_2$ in terms of the parameters $\alpha$ and $\beta$. 
Indeed, by combining the expressions of the first four moments provided in 
Lemmas~\ref{lem:jacobi-moments} and~\ref{lem:jacobi-m3m4}, and applying the formula
\begin{equation*}
    \rho_2= m_{4,\omega} - m_{2,\omega}^2 - 
\frac{\left(m_{3,\omega}-m_{1,\omega}m_{2,\omega}\right)^2}{m_{2,\omega}-m_{1,\omega}^2},
\end{equation*}
after straightforward simplifications we obtain
\begin{equation}\label{rhosimpalpha}
\rho_2 =
\frac{32(\alpha+1)(\alpha+2)(\beta+1)(\beta+2)}
{(\alpha+\beta+2)(\alpha+\beta+3)^2(\alpha+\beta+4)^2(\alpha+\beta+5)}.
\end{equation}
\end{remark}

We now give the explicit expression of the basis functions associated with the triple  $\mathcal{H}^{\alpha,\beta}_2$ relative to the Jacobi density function.  

\begin{theorem}\label{th:basis-jacobi-explicit}
Let $\omega_{\alpha,\beta}$ be the Jacobi-type probability density function.
Then, the basis functions associated to the triple $\mathcal{H}^{\alpha,\beta}_2$ relative to $\omega_{\alpha,\beta}$ are given by 
\begin{align*}
\varphi_j &= \sigma_0\lambda_j
+ \sigma_+\lambda_{j+1}
+ \sigma_-\lambda_{j+2},
\quad j=1,2,3,\\[0.4em]
\psi_j &= \nu_0(\alpha,\beta)\lambda_j
+ \nu_+\lambda_{j+1}
+ \nu_-\lambda_{j+2}
- \frac{4}{\rho_2}\lambda_{j+1}\lambda_{j+2},
\quad j=1,2,3,
\end{align*}
where 
\begin{align*}
\sigma_0 &= -\frac{(\alpha+1)(\beta+1)}{\alpha^2-\alpha\beta+\beta^2+\alpha+\beta+1},\\[0.3em]
\sigma_+ &= \frac{(\beta+1)^2}{\alpha^2-\alpha\beta+\beta^2+\alpha+\beta+1},\\
\sigma_-&= \frac{(\alpha+1)^2}{\alpha^2-\alpha\beta+\beta^2+\alpha+\beta+1},\\
\tau
&= \frac{(\alpha+\beta+2)^2(\alpha+\beta+3)(\alpha+\beta+4)^2(\alpha+\beta+5)}
       {32(\alpha+2)(\beta+2)\left(\alpha^2-\alpha\beta+\beta^2+\alpha+\beta+1\right)},\\[0.4em]
\nu_0 &= -\frac{(\alpha+\beta+3)(\alpha+\beta+4)^2(\alpha+\beta+5)(\alpha+1)(\beta+1)}
                              {8(\alpha+2)(\beta+2)\left(\alpha^2-\alpha\beta+\beta^2+\alpha+\beta+1\right)},\\[0.3em]
\nu_+ &= \frac{(\alpha+\beta+3)(\alpha+\beta+4)^2(\alpha+\beta+5)(\beta+1)^2}
                             {8(\alpha+2)(\beta+2)\left(\alpha^2-\alpha\beta+\beta^2+\alpha+\beta+1\right)},\\[0.3em]
\nu_- &= \frac{(\alpha+\beta+3)(\alpha+\beta+4)^2(\alpha+\beta+5)(\alpha+1)^2}
                             {8(\alpha+2)(\beta+2)\left(\alpha^2-\alpha\beta+\beta^2+\alpha+\beta+1\right)}
\end{align*}
and $\rho_2$ is given in~\eqref{rhosimpalpha}. 
\end{theorem}
\begin{proof}
 The claim follows directly from Lemmas~\ref{lem:jacobi-moments} and~\ref{lem:jacobi-m3m4} together with Theorem~\ref{newth}.  
\end{proof}

Then, we define the reconstruction operator associated to the triple $\mathcal{H}^{\alpha,\beta}_2$ relative to the Jacobi density function as
\begin{equation}\label{pi11}
\begin{array}{rcl}
\Pi_{2}^{\alpha,\beta}\colon C(T) &\longrightarrow& \mathbb{P}_{2}(T) \\[1ex]
f &\longmapsto& \displaystyle\sum_{j=1}^{3} 
   \mathcal{I}^{\omega}_{j}(f)\varphi_{j}+\sum_{j=1}^{3} 
   \mathcal{L}^{\omega,2}_{j}(f)\psi_{j}. 
\end{array}
\end{equation}

\subsection{Gegenbauer-type probability density}

A particularly relevant situation arises when the parameters of the Jacobi-type
probability density coincide, i.e.\ $$\alpha=\beta=\gamma.$$ In this case the
distribution is symmetric with respect to the origin, and will be referred to
as a \emph{Gegenbauer-type probability density}. The probability density function~\eqref{jacobi-normalized} then reduces to 
\begin{equation}\label{gegen-normalized}
\omega_{\gamma}(t)=\frac{(1-t^2)^{\gamma}}
     {2^{2\gamma+1}B(\gamma+1,\gamma+1)}.
\end{equation}
 \begin{figure}
      \centering
\includegraphics[width=0.49\linewidth]{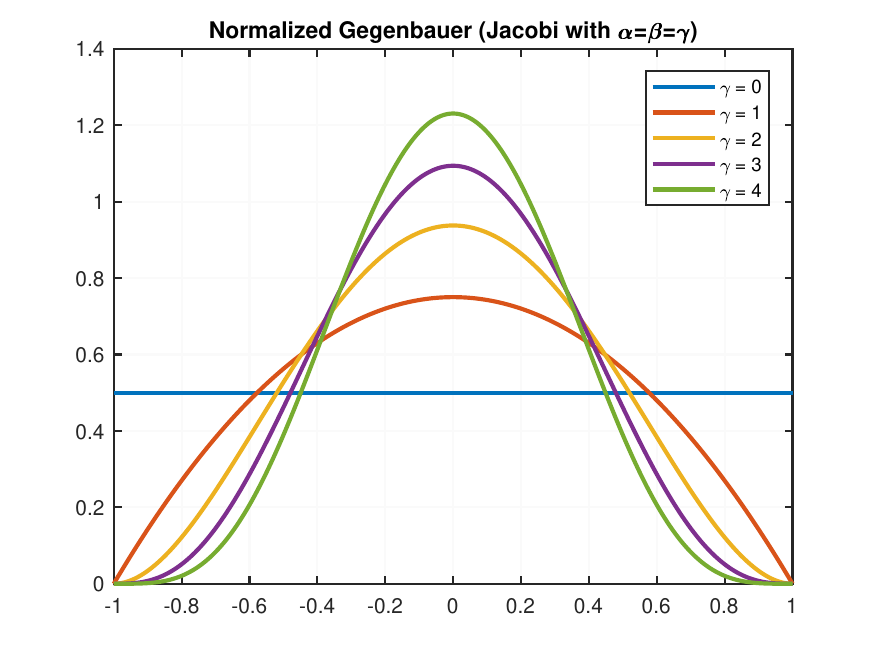}
      \caption{Normalized Gegenbauer probability density functions $\omega_{\gamma}$ on $[-1,1]$
for different values of $\gamma$. As $\gamma$ increases, the density becomes more concentrated around
$t=0$ and vanishes more rapidly near the endpoints $t=\pm 1$. This illustrates how the single
parameter $\gamma$ controls the balance between central concentration and endpoint weight.}
      \label{fig:densities1}
  \end{figure}
As shown in Figure~\ref{fig:densities1}, increasing $\gamma$ enhances the sharpness of the distribution near the center and reduces its mass near the endpoints, thus providing a simple and flexible symmetric weighting mechanism.

The symmetry immediately
implies the vanishing of all odd moments, i.e.
\begin{equation*}
m_{1,\omega}=m_{3,\omega}=0,    
\end{equation*}
which leads to a substantial simplification of the coefficients appearing in
the inverse matrix~\eqref{invmat}. In particular, the parameters reduce to
\begin{equation*}
\sigma_0 = -1, 
\quad \sigma_+ = \sigma_- = 1, 
\quad \nu_0 = -\tau, 
\quad \nu_+ = \nu_- = \tau,
\end{equation*}
where
\begin{equation}\label{tauerho}
\tau = \frac{1-m_{2,\omega}}{\rho_2}, 
\quad 
\rho_2 = m_{4,\omega}-m_{2,\omega}^2.    
\end{equation}

Consequently, the basis functions $\varphi_j,\psi_j$, $j=1,2,3,$ introduced in
Theorem~\ref{newth} exhibit a symmetric and compact representation
\begin{equation*}
    \varphi_j = 1-2\lambda_j, 
\quad
\psi_j = \tau \varphi_j - \frac{4}{\rho_2}\lambda_{j+1}\lambda_{j+2}, 
\quad j=1,2,3,
\end{equation*}
which considerably simplifies both the theoretical analysis and the numerical
implementation of the scheme.

\medskip
Moreover, in this case the expressions of $\tau$ and $\rho_2$ admit closed forms that
depend only on the parameter $\gamma$. Indeed, by substituting
$\alpha=\beta=\gamma$ into~\eqref{tauerho}, we obtain
\begin{equation*}
\tau = \frac{1-m_{2,\omega}}{\rho_2}
= \frac{(2\gamma+3)(2\gamma+5)}{2},
\end{equation*}
and
\begin{equation*}
\rho_2 = m_{4,\omega} - m_{2,\omega}^2
= \frac{4(\gamma+1)}{(2\gamma+3)^2(2\gamma+5)}.
\end{equation*}

\begin{remark}
The Gegenbauer-type family defines a natural one-parameter model depending only
on $\gamma$. It includes the uniform distribution on $[-1,1]$ as the special
case $\gamma=0$, while for $\gamma>0$ it smoothly interpolates between different
symmetric weighting patterns. This balance between simplicity and flexibility
makes the Gegenbauer case particularly attractive for practical applications.
\end{remark}

\section{Parameter selection for optimal reconstruction}
\label{sec3}

We now address the problem of selecting the edge--density parameters so as to
optimize the overall reconstruction accuracy on a prescribed validation set.
We present a mesh--level (global) tuning strategy based on grid search, which
is simple, robust, and easily parallelizable.

\subsection{Global tuning on the mesh: Jacobi case $(\alpha,\beta)$}
\label{sec:param-jacobi}

Let $\left\{f^{(r)}\right\}_{r=1}^{\widetilde{R}}$ be a validation collection of target functions and
$\left\{\mathcal T_n\right\}_{n=0}^{\widetilde{N}}$ a sequence of meshes on $\Omega$. For the Jacobi
edge densities $\omega_{\alpha,\beta}$, defined in~\eqref{jacobi-normalized}, with admissible
parameters $\alpha,\beta>-1$, we determine
$(\alpha^\star,\beta^\star)$ by minimizing the global validation error over a
finite candidate grid $\mathcal P\subset(-1,\infty)^2$, that is
\begin{equation}
\label{eq:globalJac}
(\alpha^\star,\beta^\star)\in
\arg\min_{(\alpha,\beta)\in\mathcal P}
\sum_{r=1}^{\widetilde{R}}\sum_{n=0}^{\widetilde{N}}
\left\| f^{(r)} -
\Pi_{2}^{\alpha,\beta}\left[f^{(r)};\mathcal T_n\right]
\right\|_{L^1(\Omega;\mathcal T_n)}.
\end{equation}
Here $\Pi_{2}^{\alpha,\beta}$ denotes the quadratic enriched operator defined in~\eqref{pi11}.
Other norms (e.g., $L^2$ or $L^\infty$) or mixed criteria can be accommodated
with no change in the algorithmic structure.

\begin{algorithm}[H]
\caption{Global mesh--level tuning for Jacobi densities}
\label{alg:grid-jacobi}
\begin{algorithmic}[1]
\Require Validation functions $\{f^{(r)}\}_{r=1}^{\widetilde{R}}$; meshes $\{\mathcal T_n\}_{n=0}^{\widetilde{N}}$;
candidate grid $\mathcal P\subset(-1,\infty)^2$ of parameter pairs $(\alpha,\beta)$;
reconstruction operator $\Pi_{2}^{\alpha,\beta}$
\Ensure Optimal parameters $(\alpha^\star,\beta^\star)$
\State $E_{\min}\gets+\infty$;\quad $(\alpha^\star,\beta^\star)\gets\text{undefined}$
\ForAll{$(\alpha,\beta)\in\mathcal P$}
  \State $E(\alpha,\beta)\gets 0$
  \For{$r=1$ to $\widetilde{R}$}
    \For{$n=0$ to ${\widetilde{N}}$}
      \State $u \gets \Pi_{2}^{\alpha,\beta}[f^{(r)};\mathcal T_n]$
      \State $E(\alpha,\beta) \gets E(\alpha,\beta) + \|f^{(r)}-u\|_{L^1(\Omega;\mathcal T_n)}$
    \EndFor
  \EndFor
  \If{$E(\alpha,\beta) < E_{\min}$}
     \State $E_{\min}\gets E(\alpha,\beta)$;\quad $(\alpha^\star,\beta^\star)\gets(\alpha,\beta)$
  \EndIf
\EndFor
\State \Return $(\alpha^\star,\beta^\star)$
\end{algorithmic}\label{alg1}
\end{algorithm}

\section{Numerical experiments}\label{sec4}
In this section we provide a set of numerical experiments aimed at assessing the practical
performance of the proposed enriched histopolation framework. The tests are performed on
the square domain $\Omega = [-1,1]^2$, employing a representative collection of benchmark
functions that display different structural features, including smoothness, oscillations,
singular behavior, and localized peaks. Specifically, we consider the following test functions
 \begin{eqnarray*}
    f_1(x,y)&=&\sqrt{x^2+y^2}, \ \ f_2(x,y)=e^{-4\left(x^2+y^2\right)}\sin\left(\pi(x+y)\right),\\
    f_3(x,y)&=&\sin(2\pi x)\sin(2\pi y),\ \, 
    f_4(x,y)=\sin\left(4\pi(x+y)\right),\ \,
    f_5(x,y)=\frac{1}{25(x^2+y^2)+1},\\ 
 f_6(x,y)&=&0.75\exp\biggl(-\frac{(9(x+1)/2-2)^2}{4}-\frac{(9(y+1)/2-2)^2}{4}\biggr)\\[-2pt]
 &&+0.75\exp\biggl(-\frac{(9(x+1)/2+1)^2}{49}-\frac{(9(y+1)/2+1)}{10}\biggr)\\[-2pt]
	&&+ 0.5\exp\biggl(-\frac{(9(x+1)/2-7)^2}{4}-\frac{(9(y+1)/2-3)^2}{4}\biggr)\\
	&&-0.2\exp\biggl(-(9(x+1)/2-4)^2-(9(y+1)/2-7)^2\biggr).
\end{eqnarray*}
The function $f_6$ is the well-known Franke function, widely used as a benchmark for approximation methods~\cite{franke1982scattered}. 
For the spatial discretization we employ a sequence of regular Friedrichs--Keller
triangulations~\cite{Knabner}, denoted by
\begin{equation*}
    \mathcal{T}_n=\left\{T_i : i=1,\dots,2(n+1)^2\right\},
\end{equation*}
so that each $\mathcal{T}_n$ consists of $2(n+1)^2$ congruent triangles; see
Figure~\ref{Fig:rec0}.
\begin{figure}
    \centering
\includegraphics[scale=0.49]{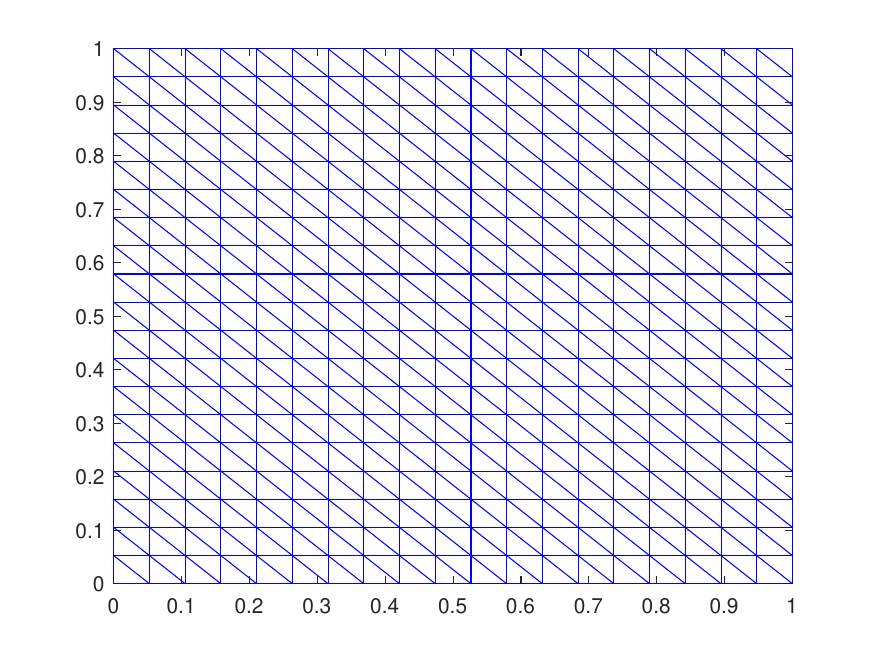}
\includegraphics[scale=0.49]{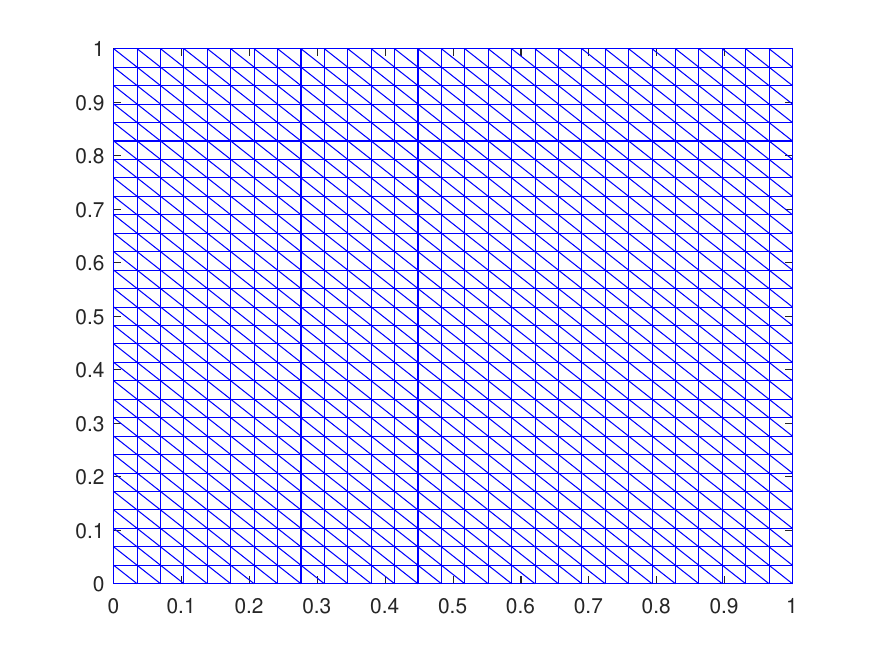}
\includegraphics[scale=0.49]{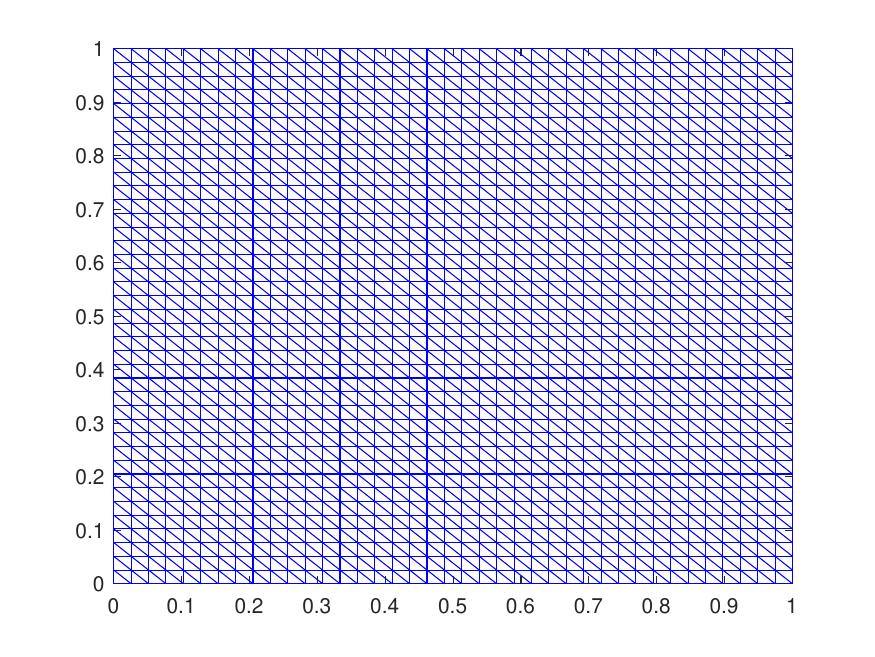}
\includegraphics[scale=0.49]{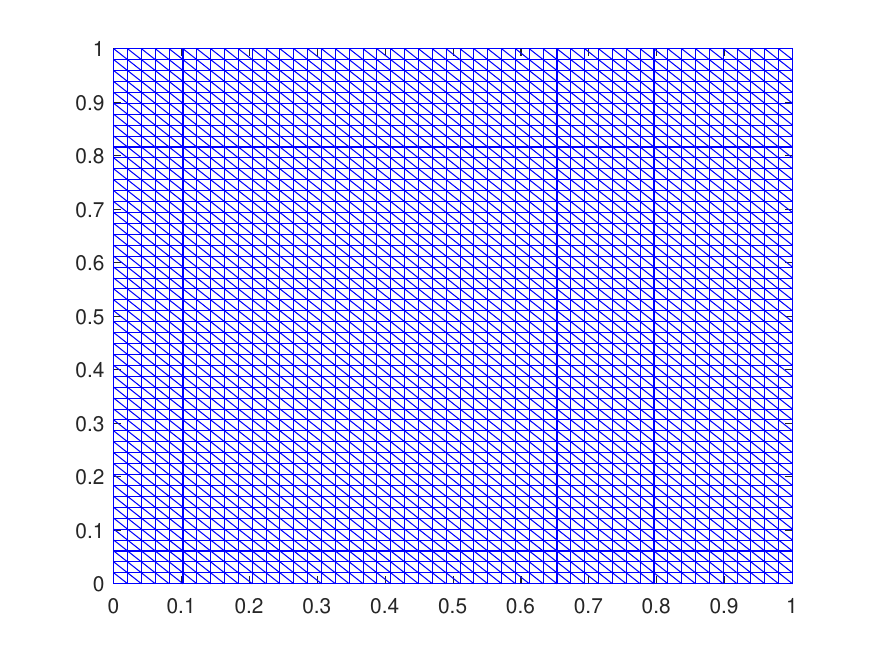}
    \caption{Examples of regular Friedrichs--Keller triangulations $\mathcal{T}_n$, for
$n=20$ (top left), $n=30$ (top right), $n=40$ (bottom left), and $n=50$ (bottom right).}
    \label{Fig:rec0}
\end{figure}
We compare the quality of the reconstruction in terms of the $L^1$ error norm of the two schemes:  
\begin{itemize}
    \item the classical piecewise linear nonconforming histopolation method
$\mathcal{CH}$, 
\item the enriched quadratic histopolation operator $\mathcal{H}^{\alpha^\star,\beta^\star}_2$,
based on Jacobi edge densities with optimally tuned parameters $(\alpha^\star,\beta^\star)$
selected through Algorithm~\ref{alg1}. 
\end{itemize}
The choice of the $L^1$ norm is motivated by its robustness in the presence of
discontinuities and localized features, without requiring additional stabilization
(cf.~\cite{dell2025truncated}). 
The results of this comparison are reported in Figure~\ref{im12}.

\begin{figure}[h]
  \centering
   \includegraphics[width=0.49\textwidth]{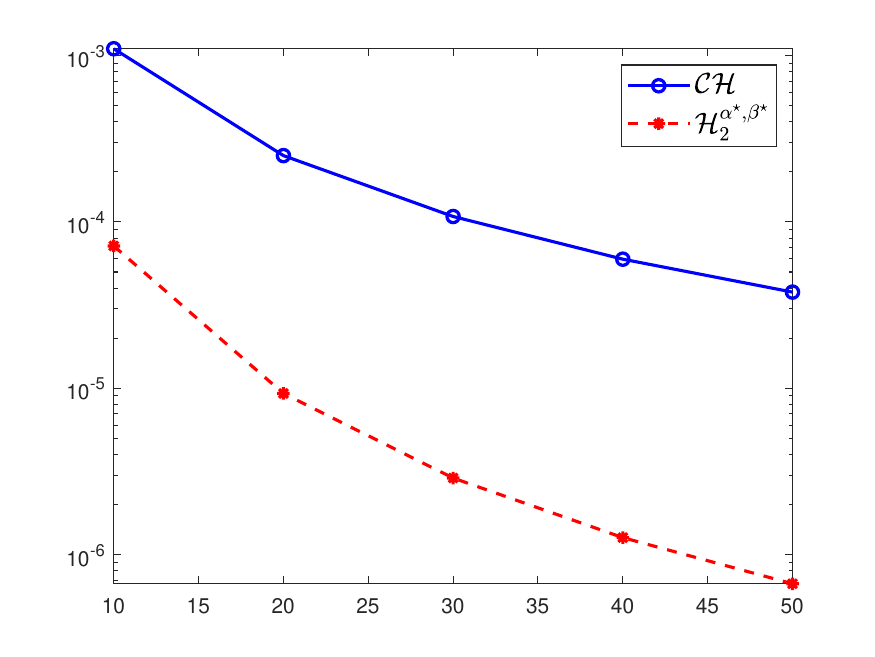} 
\includegraphics[width=0.49\textwidth]{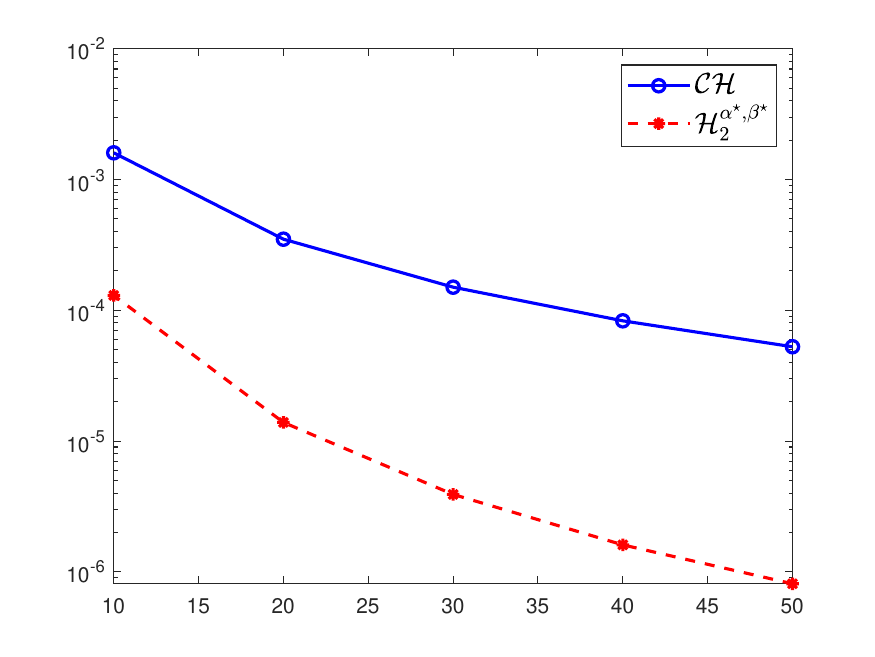} 
\includegraphics[width=0.49\textwidth]{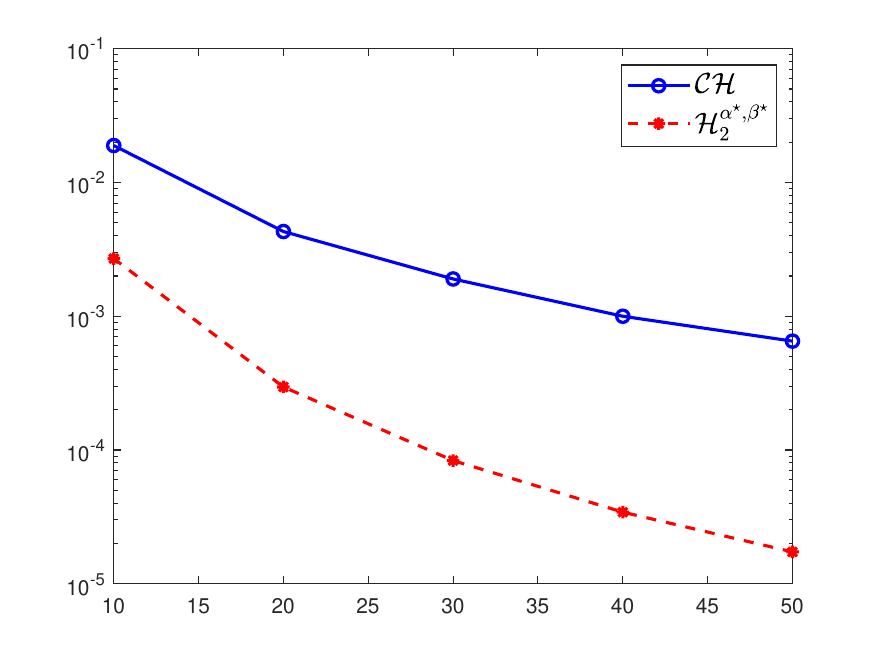} 
\includegraphics[width=0.49\textwidth]{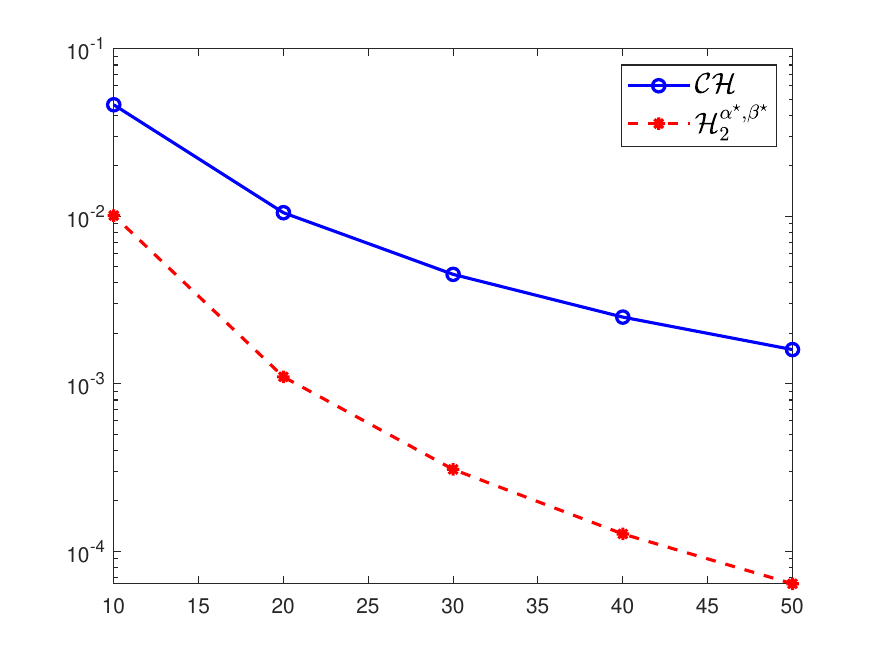}
\includegraphics[width=0.49\textwidth]{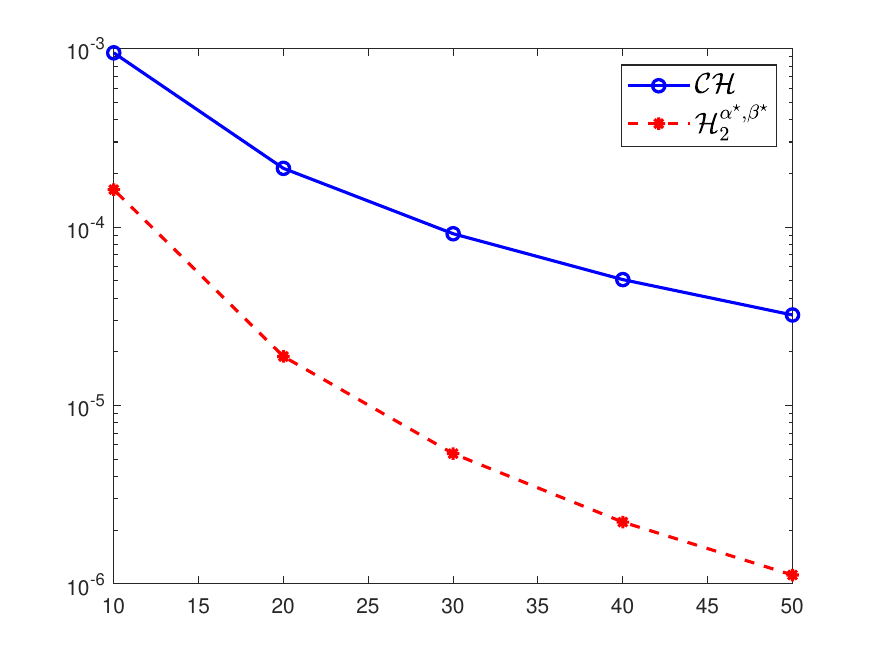} 
\includegraphics[width=0.49\textwidth]{f5.pdf}
         \caption{Semi-log plot of the $L^1$ approximation error for 
   $f_1$ (top left), $f_2$ (top right), 
   $f_3$ (middle left), $f_4$ (middle right), 
   $f_5$ (bottom left), and $f_6$ (bottom right). 
   Comparison between the classical histopolation method $\mathcal{CH}$ (blue) 
   and the enriched Jacobi-based method 
$\mathcal{H}^{\alpha^\star,\beta^\star}_2$ (red), 
   as the number of triangles in the Friedrichs--Keller triangulations increases.}
\label{im12}
\end{figure}

A clear improvement of the enriched operator over the classical scheme is consistently observed 
for all benchmark functions and mesh refinements. In particular, 
$\mathcal{H}^{\alpha^\star,\beta^\star}_2$ yields substantially smaller $L^1$ errors, 
highlighting its superior ability to reproduce local oscillations, sharp gradients, 
and singular behaviors. Furthermore, the enriched approach exhibits a faster decay of the 
approximation error under mesh refinement, confirming that the additional degrees of freedom 
introduced by the Jacobi-type probability density weights translate into a genuine 
enhancement of the approximation power.

\section{Conclusions and future works}
\label{sec5}
In this work we have developed a generalized probability density framework for local histopolation on triangular meshes, with a focus on Jacobi-type edge weights and their Gegenbauer subclass. We established the theoretical foundations of the enriched quadratic operators, proving unisolvency and deriving explicit basis functions, and we designed a parameter tuning strategy that ensures robustness and adaptivity. The numerical experiments confirmed the clear superiority of the proposed approach over the classical linear histopolation scheme, both in terms of accuracy and error decay under mesh refinement.  

Several directions of future research are currently under consideration.  On the application side, the flexibility of the Jacobi and Gegenbauer densities makes the proposed framework particularly promising for problems in imaging, inverse problems, and numerical simulations where weighted integral data naturally arise. These perspectives suggest that enriched histopolation methods based on probability densities can become a versatile tool for accurate and adaptive function reconstruction in a wide range of scientific and engineering applications.

\section*{Acknowledgments}
This research has been achieved as part of RITA \textquotedblleft Research
 ITalian network on Approximation'' and as part of the UMI group \enquote{Teoria dell'Approssimazione
 e Applicazioni}. The research was supported by GNCS–INdAM 2025 project \emph{``Polinomi, Splines e Funzioni Kernel: dall'Approssimazione Numerica al Software Open-Source''}. 
 The  research of G.V. Milovanovi\'c   was supported in part by the Serbian Academy of Sciences and Arts, Belgrade (Grant No. $\Phi$-96). 
The work of F. Nudo is funded from the European Union – NextGenerationEU under the Italian National Recovery and Resilience Plan (PNRR), Mission 4, Component 2, Investment 1.2 \lq\lq Finanziamento di progetti presentati da giovani ricercatori\rq\rq,\ pursuant to MUR Decree No.~47/2025.

\section*{Conflict of interest}
Not Applicable.

\bibliographystyle{elsarticle-num}
\bibliography{bibliografia.bib}
\end{document}